\documentclass[oneside,english,reqno]{amsart}
\usepackage{lmodern}
\usepackage[T1]{fontenc}
\usepackage[latin9]{inputenc}
\usepackage{babel}
\usepackage{float}
\usepackage{amsthm}
\usepackage{amssymb}
\usepackage[unicode=true,pdfusetitle,
 bookmarks=true,bookmarksnumbered=false,bookmarksopen=false,
 breaklinks=false,pdfborder={0 0 1},backref=false,colorlinks=false]
 {hyperref}

\makeatletter

\floatstyle{ruled}
\newfloat{algorithm}{tbp}{loa}
\providecommand{\algorithmname}{Algorithm}
\floatname{algorithm}{\protect\algorithmname}

\numberwithin{equation}{section}
\numberwithin{figure}{section}
\theoremstyle{plain}
\newtheorem{thm}{\protect\theoremname}[section]
  \theoremstyle{plain}
  \newtheorem{lyxalgorithm}[thm]{\protect\algorithmname}
  \theoremstyle{plain}
  \newtheorem{prop}[thm]{\protect\propositionname}
  \theoremstyle{plain}
  \newtheorem{assumption}[thm]{\protect\assumptionname}
  \theoremstyle{remark}
  \newtheorem{claim}[thm]{\protect\claimname}
  \theoremstyle{remark}
  \newtheorem{rem}[thm]{\protect\remarkname}

\makeatother

  \providecommand{\algorithmname}{Algorithm}
  \providecommand{\assumptionname}{Assumption}
  \providecommand{\claimname}{Claim}
  \providecommand{\propositionname}{Proposition}
  \providecommand{\remarkname}{Remark}
\providecommand{\theoremname}{Theorem}

\begin{document}
\title[Parallel Dykstra splitting]{A general framework for parallelizing Dyskstra splitting} 

\begin{abstract}
We show a general framework of parallelizing Dykstra splitting that
includes the classical Dykstra's algorithm and the product space formulation
as special cases, and prove their convergence. The key idea is to
split up the function whose conjugate takes in the sum of all dual
variables in the dual formulation.
\end{abstract}

\author{C.H. Jeffrey Pang}

\thanks{We acknowledge grant R-146-000-214-112 from the Faculty of Science,
National University of Singapore. }

\curraddr{Department of Mathematics\\ 
National University of Singapore\\ 
Block S17 08-11\\ 
10 Lower Kent Ridge Road\\ 
Singapore 119076 }

\email{matpchj@nus.edu.sg}

\date{\today{}}


\maketitle
\tableofcontents{}

\section{Introduction}

Let $X$ be a finite dimensional Hilbert space. Consider the problem
\begin{equation}
(P1)\qquad\min_{x\in X}\sum_{i=1}^{r}\left(h_{i}(x)\right)+g(x),\label{eq:first-primal}
\end{equation}
where $h_{i}:X\to\mathbb{R}$ and $g:X\to\mathbb{R}$ are proper closed
convex functions. The (Fenchel) dual of $(P1)$ is 
\begin{equation}
(D1)\qquad\max_{z\in X^{r}}-\sum_{i=1}^{r}\left(h_{i}^{*}(z_{i})\right)-g^{*}\left(-\sum_{i=1}^{r}z_{i}\right).\label{eq:dual-1}
\end{equation}
A particular case of $(P1)$ that is well studied is when $g(x):=\frac{1}{2}\|x-x_{0}\|^{2}$
for some $x_{0}\in X$, and 
\[
g^{*}(y)=\frac{1}{2}\|y\|^{2}+x_{0}^{T}y=\frac{1}{2}\|y+x_{0}\|^{2}-\frac{1}{2}\|x_{0}\|^{2}.
\]
The resulting $(D1)$ would be the sum of a block separable concave
function and a smooth concave function. For the problem $(D1)$, if
the map $z\mapsto g^{*}(-\sum_{i=1}^{r}z_{i})$ is smooth, then $(D1)$
can be solved by block coordinate minimization (BCM); Specifically,
one maximizes a particular $z_{i}$, say $z_{i^{*}}$, while keeping
all other $z_{j}$, where $j\neq i^{*}$, fixed, and the index $i^{*}$
is cycled over all indices in $\{1,\dots,r\}$. (It would be a minimization
if $(D1)$ were written in a minimization form.)

Dykstra's algorithm was proposed in \cite{Dykstra83}, and it was
separately recognized to be the BCM on $(D1)$ with $h_{i}(\cdot)\equiv\delta_{C_{i}}(\cdot)$
for all $i\in\{1,\dots,r\}$ and $g(x):=\frac{1}{2}\|x-x_{0}\|^{2}$
in \cite{Han88,Gaffke_Mathar}. An advantage of Dykstra's algorithm
is that it decomposes the complicated problem $(P1)$ so that the
proximal operation is applied to only one function of the form $h_{i}(\cdot)$
(or $h_{i}^{*}(\cdot)$) at a time so that one can solve the larger
problem in hand. Dykstra's algorithm was shown to converge to the
primal minimizer in \cite{BD86}, even when a dual minimizer does
not exist. For more information Dykstra's algorithm, we refer to \cite{Deutsch01_survey,Deustch01,BauschkeCombettes11,EsRa11}.

The extension of considering general $h_{i}(\cdot)$ was done in \cite{Han89}
and \cite{Tseng-93-Dual-ascent}. We now refer to this as Dykstra's
splitting. It is quite easy to see that the convergence to the dual
objective value implies the convergence to the primal minimizer. (See
for example, the end of the proof in Theorem \ref{thm:convergence}.)
In \cite{Han89} and \cite{Tseng-93-Dual-ascent}, they proved that
Dykstra's splitting converges, but under a constraint qualification.
The paper \cite{Bauschke-Combettes-Dykstra-split-2008} proved that
Dykstra's splitting converges without constraint qualifications, but
only for the case where $r=2$. In \cite{Pang_Dyk_spl}, we proved
the convergence of Dykstra's splitting in finite dimensions for any
$r\geq2$.

The BCM algorithm is related to block coordinate gradient descent,
but we shall only mention them in passing as we do not deal directly
with these algorithms in this paper. Much research on the BCM and
related algorithms is on nonasymptotic convergence rates when a minimizer
to $(D1)$ exists and the level sets are bounded.

A shortcoming of Dykstra's splitting is that it requires that the
proximal operations on $h_{i}(\cdot)$ be taken one at a time in order.
If data on the functions $h_{i}(\cdot)$ were distributed on different
agents, then these agents would be idling as they wait for their turn.
Another parallel method for solving $(P1)$ when $r>2$ is to use
the product space formulation largely due to \cite{Pierra84}. (See
the first paragraph in Section \ref{sec:pdt-space-form} for more
details.) A shortcoming of this product space formulation is that
a central controller needs to compute the average of all intermediate
primal variables before the next iteration can proceed. This can be
a tedious task depending on the communication model, and if it were
easy to do, then accelerated proximal algorithms might be preferred
(see for example \cite{BeckTeboulle2009,Tseng_APG_2008}, who built
on the work of \cite{Nesterov_1983}). Parallelizations of the BCM
were suggested in \cite{Calafiore_2016_parallel,Richtarik_Takac_parallel},
but we note that their approach is different from what we will discuss
in this paper, and the approach in \cite{Richtarik_Takac_parallel}
requires random sampling.

Dykstra's splitting falls under the larger class of proximal methods.
See a survey in \cite{Combettes_Pesquet}.

\subsection{Contributions of this paper}

The contribution of this paper is to parallelize the BCM problem arising
from Dykstra's algorithm so that agents that otherwise would have
been idle in Dykstra's algorithm can be actively decreasing the dual
objective value. This is achieved by breaking up the $g(\cdot)$ in
$(P1)$ so that smaller versions of the problem of the form $(D1)$
can be solved in parallel with few communication requirements between
the agents involved. We also show that our algorithm generalizes the
product space formulation. We also prove its convergence to the primal
minimizer in the spirit of \cite{BD86,Gaffke_Mathar}, even when a
dual minimizer may not exist.

\section{Algorithm description}

Consider the problem $(P1)$. Let $\{\lambda_{i}\}_{i=0}^{m}$ be
constants such that $\sum_{i=0}^{m}\lambda_{i}=1$ and $\lambda_{i}\geq0$.
Define $h_{j}:X\to\mathbb{R}$ to be 
\begin{equation}
\begin{array}{c}
h_{j}(x)=\lambda_{j-r}g(x)\mbox{ for all }j\in\{r+1,\dots,r+m\}.\end{array}\label{eq:def-f-j}
\end{equation}
Then \eqref{eq:first-primal} can be rewritten as 
\begin{equation}
(P2)\qquad\min_{x}\,\,\sum_{i=1}^{r+m}h_{i}(x)+\lambda_{0}g(x),\label{eq:second-primal}
\end{equation}
which in turn has dual 
\begin{equation}
(D2)\qquad\max_{z\in X^{r+m}}\,\,F(z):=-\sum_{i=1}^{r+m}h_{i}^{*}(z_{i})-\lambda_{0}g^{*}\left(-\frac{1}{\lambda_{0}}\sum_{i=1}^{r+m}z_{i}\right),\label{eq:dual}
\end{equation}
where 
\begin{equation}
h_{i}^{*}(z_{i})\overset{\eqref{eq:def-f-j}}{=}\lambda_{i-r}g^{*}\left(\frac{z_{i}}{\lambda_{i-r}}\right)\mbox{ for all }i\in\{r+1,\dots,r+m\}.\label{eq:f-j-star}
\end{equation}
Algorithm \ref{alg:Dist-Dyk} describes a method to solve $(D2)$
(and equivalently, $(P1)$ and $(P2)$). 

\begin{algorithm}[th]
\begin{lyxalgorithm}
\label{alg:Dist-Dyk}(A general framework for Dykstra's splitting)
This algorithm solves \eqref{eq:dual}, which can in turn find the
primal minimizer of \eqref{eq:first-primal}. Let $\bar{w}$ be a
fixed number.

01$\quad$Start with dual variables $\{z_{i}^{1,0}\}_{i=1}^{r+m}$
to the problem \eqref{eq:dual}.

02$\quad$For $n=1,\dots$

03$\quad$$\quad$For $w=1,\dots,\bar{w}$

04$\quad$$\quad$$\quad$Choose a (possibly empty) subset $S_{n,w}\subset\{1,\dots,r+m\}$.

05$\quad$$\quad$$\quad$Let $\{z_{i}^{n,w}\}_{i\in S_{n,w}}$ be
a minimizer of 

\begin{equation}
\begin{array}{c}
\underset{z_{i}:i\in S_{n,w}}{\min}\,\,\underset{i\in S_{n,w}}{\sum}h_{i}^{*}(z_{i})+\lambda_{0}g^{*}\bigg(-\frac{1}{\lambda_{0}}\bigg[\underset{i\in S_{n,w}}{\sum}z_{i}+\underset{i\notin S_{n,w}}{\sum}z_{i}^{n,w-1}\bigg]\bigg)\end{array}\label{eq:outer-opt}
\end{equation}

06$\quad$$\quad$$\quad$Let $\{z_{i}^{n,w}\}_{i\notin S_{n,w}}$
be such that\begin{subequations}\label{eq_m:inner-opt} 
\begin{eqnarray}
\underset{i\notin S_{n,w}}{\sum}h_{i}^{*}(z_{i}^{n,w}) & \leq & \underset{i\notin S_{n,w}}{\sum}h_{i}^{*}(z_{i}^{n,w-1})\label{eq:inner-opt-1}\\
\mbox{ and }\underset{i\notin S_{n,w}}{\sum}z_{i}^{n,w} & = & \underset{i\notin S_{n,w}}{\sum}z_{i}^{n,w-1}.\label{eq:inner-opt-2}
\end{eqnarray}

\end{subequations}07$\quad$$\quad$End for

08$\quad$$\quad$Let $z^{n+1,0}=z^{n,\bar{w}}$.

09$\quad$End for \end{lyxalgorithm}
\end{algorithm}

Dykstra's algorithm is usually expressed in terms of the primal variables.
We refer to Proposition \ref{prop:primal-form}. 

We note that in Algorithm \ref{alg:Dist-Dyk}, both problems \eqref{eq:outer-opt}
and \eqref{eq_m:inner-opt} are aimed to increase the dual objective
value in \eqref{eq:dual}.

We show a direct connection between $(D1)$ and $(D2)$, whose proof
is just elementary convexity.
\begin{prop}
\label{prop:direct-D1-D2}(Direct connection between $(D1)$ and $(D2)$)
Consider the problem 
\begin{equation}
\min_{z_{i}:r+1\leq i\leq r+m}\lambda_{0}g^{*}\left(-\frac{1}{\lambda_{0}}\left(\sum_{i=1}^{r}\bar{z}_{i}+\sum_{i=r+1}^{r+m}z_{i}\right)\right)+\sum_{i=r+1}^{r+m}\underbrace{\lambda_{i-r}g^{*}\left(\frac{1}{\lambda_{i-r}}z_{i}\right)}_{\overset{\eqref{eq:f-j-star}}{=}h_{i}^{*}(z_{i})}.\label{eq:agg}
\end{equation}
Suppose further that $g^{*}(\cdot)$ is strictly convex. Then \eqref{eq:agg}
has a minimum value of $g^{*}\left(-\sum_{i=1}^{r}\bar{z}_{i}\right)$
with minimizer 
\begin{equation}
z_{j}=-\lambda_{j-r}\sum_{i=1}^{r}\bar{z}_{i}\mbox{ for all }j\in\{r+1,\dots,r+m\}.\label{eq:z-j-conclusion}
\end{equation}
This establishes the equivalence of $(D1)$ and $(D2)$ directly
without appealing to the primal problems $(P1)$ and $(P2)$.
\end{prop}

\subsection{\label{sub:general-strategies}Finding $\{z_{i}^{n,w}\}_{i\protect\notin S_{n,w}}$
in \eqref{eq_m:inner-opt}}

We now suggest methods for finding $\{z_{i}^{n,w}\}_{i\notin S_{n,w}}$
satisfying \eqref{eq_m:inner-opt}. Let $S'$ be a subset of $\{1,\dots,r+m\}\backslash S_{n,w}$,
and suppose $j\in\{r+1,\dots,r+m\}\backslash S'$. Consider the problem
\begin{eqnarray}
 & \underset{z_{i}:i\in S'\cup\{j\}}{\min} & \sum_{i\in S'\cup\{j\}}h_{i}^{*}(z_{i})\label{eq:typical-ADMM}\\
 & \mbox{s.t. } & \sum_{i\in S'\cup\{j\}}z_{i}=\sum_{i\in S'\cup\{j\}}z_{i}^{n,w-1}.\nonumber 
\end{eqnarray}
Such a problem can be solved by other methods like the ADMM. But intermediate
iterates of the other methods may not satisfy the equality constraint
of \eqref{eq:typical-ADMM}, so one has to check whether these intermediate
iterates are indeed more useful than what we had started off with.
We now show an option that is close to the spirit of BCM. 
\begin{prop}
\label{prop:reduced-pblm}(Reduced unconstrained problem) Let $j\in\{r+1,\dots,r+m\}$,
and suppose $j\notin S'$. Recall that $h_{j}(x)\overset{\eqref{eq:def-f-j}}{=}\lambda_{j-r}g(x)$.
For the problem \eqref{eq:typical-ADMM}, $\{z_{i}\}_{i\in S'\cup\{j\}}$
is a minimizer to \eqref{eq:typical-ADMM} if and only if $\{z_{i}\}_{i\in S'}$
is a minimizer to 
\begin{equation}
\begin{array}{c}
\underset{z_{i}:i\in S'}{\min}\,\,\underset{i\in S'}{\sum}h_{i}^{*}(z_{i})+\lambda_{j-r}g^{*}\bigg(\frac{1}{\lambda_{j-r}}\bigg[\underset{i\in S'\cup\{j\}}{\sum}z_{i}^{n,w-1}-\underset{i\in S'}{\sum}z_{i}\bigg]\bigg).\end{array}\label{eq:transformed-dual}
\end{equation}
\end{prop}
\begin{proof}
The linear constraint in \eqref{eq:typical-ADMM} can be removed by
expressing $z_{j}$ in terms of the other variables.
\end{proof}
Similar to the relationship between $(P1)$ and $(D1)$, \eqref{eq:transformed-dual}
has dual 
\begin{equation}
\begin{array}{c}
\underset{x}{\min}\,\,\underset{i\in S'}{\sum}h_{i}(x)+\lambda_{j-r}g(x)-\bigg(\underset{i\in S'\cup\{r+1\}}{\sum}z_{i}^{n,w-1}\bigg)^{T}x\end{array}\label{eq:back-primal}
\end{equation}
This problem \eqref{eq:transformed-dual} can then be solved by BCM.
Given that $|S'|$ is now smaller compared to $|\{1,\dots,r+m\}|$,
other methods might be better. For example, an accelerated proximal
gradient method can be used if the problem size is small enough that
communication problems are less significant. Interior point methods
can be considered if the subproblem size is small enough. The set
$\{1,\dots,r+m\}\backslash S_{n,w}$ can be partitioned into at most
$m$ such subsets, with each subset taking a term of the form $\lambda_{j-r}g(x)$
so that these problems can be solved in parallel.

\section{\label{sec:pdt-space-form}Product space formulation is a subcase}

Consider the problem of projecting $x_{0}$ onto $\cap_{i=1}^{r}C_{i}$.
This problem can be equivalently formulated using the product space
formulation largely attributed to \cite{Pierra84} and also studied
in \cite{Iusem_DePierro_Han_1991}; Specifically, the projection of
$(x_{0},\dots,x_{0})\in X^{r}$ onto $\mathcal{D}\cap\mathcal{C}$,
where $\mathcal{D}\subset X^{r}$ is the set $\{(x,\dots,x):x\in X\}$
and $\mathcal{C}$ is the set $C_{1}\times\cdots\times C_{r}$, gives
$P_{\cap_{i=1}^{r}C_{i}}(x_{0})$ in each component. Dykstra's algorithm
can then be applied on this formulation, which can then be rewritten
as Algorithm \ref{alg:pdt-space-alg} below. In this section, we show
that Algorithm \ref{alg:pdt-space-alg} is a special case of Algorithm
\ref{alg:Dist-Dyk}.

\begin{algorithm}[h]
\begin{lyxalgorithm}
\label{alg:pdt-space-alg}(Product space formulation for Dykstra's
algorithm) For the problem of projecting $x_{0}$ onto $\cap_{i=1}^{r}C_{i}$,
the product space formulation gives the following algorithm. (Compare
this to \cite[page 267]{Iusem_DePierro_Han_1991})

01 $\quad$Start with dual variables $\{z_{i}^{1}\}_{i=1}^{r}$, and
let $x^{1}=x_{0}-\frac{1}{r}\sum_{i=1}^{r}z_{i}^{1}$.

02 $\quad$For $n=2,\dots$

03 $\quad$$\quad$For $i=1,\dots,r$

04 $\quad$$\quad$$\quad$$u_{i}^{n}=x^{n-1}+z_{i}^{n-1}$

05 $\quad$$\quad$$\quad$$x_{i}^{n}=P_{C_{i}}(u_{i}^{n})$

06 $\quad$$\quad$$\quad$$z_{i}^{n}=u_{i}^{n}-x_{i}^{n}$

07 $\quad$$\quad$End for

08 $\quad$$\quad$$x^{n}=\frac{1}{r}\sum_{i=1}^{r}x_{i}^{n}$.

09 $\quad$End for\end{lyxalgorithm}
\end{algorithm}

Algorithm \ref{alg:calc-inner-opt} is a particular way to solve \eqref{eq_m:inner-opt}
in Algorithm \ref{alg:Dist-Dyk}, which will be used throughout this
section and the next section.

\begin{algorithm}[h]
\begin{lyxalgorithm}
\label{alg:calc-inner-opt}(Calculating \eqref{eq_m:inner-opt}) Suppose
that \eqref{eq_m:inner-opt} in Algorithm \ref{alg:Dist-Dyk} is performed
using the following step:

01$\quad$Find disjoint subsets $\{S_{n,w,j}'\}_{j=r+1}^{r+m}$ such
that 
\begin{eqnarray*}
 &  & S_{n,w}\cap\bigg[\bigcup_{j=r+1}^{r+m}S_{n,w,j}'\bigg]=\emptyset,\\
 & \mbox{and } & \mbox{For all }j\in\{r+1,\dots,r+m\}\mbox{ }\begin{cases}
S_{n,w,j}'\subset\{1,\dots,r\}\cup\{j\},\mbox{ and }\\
S_{n,w,j}'\neq\emptyset\mbox{ implies }S_{n,w,j}'\supsetneq\{j\}.
\end{cases}
\end{eqnarray*}

02$\quad$For $j=r+1,\dots,r+m$

03$\quad$$\quad$Define $\{z_{i}^{n,w}\}_{i\in S_{n,w,j}'}$ by\begin{subequations}\label{eq_m:solve-inner-j}
\begin{eqnarray}
\{z_{i}^{n,w}\}_{i\in S_{n,w,j}'}\in & \underset{z_{i}:i\in S_{n,w,j}'}{\arg\min} & \sum_{i\in S_{n,w,j}'}h_{i}(z_{i})\label{eq:solve-inner-j-1}\\
 & \mbox{s.t. } & \sum_{i\in S_{n,w,j}'}z_{i}=\sum_{i\in S_{n,w,j}'}z_{i}^{n,w-1}.\label{eq:solve-inner-j-2}
\end{eqnarray}

\end{subequations}04$\quad$End for 

05$\quad$If $i\notin S_{n,w}\cup\bigcup_{j=r+1}^{r+m}S_{n,w,j}'$,
then $z_{i}^{n,w}=z_{i}^{n,w-1}$. \end{lyxalgorithm}
\end{algorithm}
We now present our result showing that Algorithm \ref{alg:pdt-space-alg}
is a particular case of Algorithm \ref{alg:Dist-Dyk}.
\begin{prop}
(Product space formulation is a particular case) Suppose Algorithm
\ref{alg:Dist-Dyk} is run using Algorithm \ref{alg:calc-inner-opt}
for the subproblem \eqref{eq_m:inner-opt}. Suppose the parameters
are set to be $\bar{w}=2$, $m=r-1$, $g(\cdot)=\frac{1}{2}\|\cdot-x_{0}\|^{2}$
and 
\begin{enumerate}
\item [(i)]$h_{i}(\cdot)=\delta_{C_{i}}(\cdot)$ and for all $i\in\{1,\dots,r\}$,
\item [(ii)]$\lambda_{i}=\frac{1}{r}$ for all $i\in\{0,\dots,r-1\}$,
\item [(iii)]$S_{n,1}=\{r+1,\dots,2r-1\}$ and $S_{n,2}=\{r\}$ for all
$n\geq0$, and 
\item [(iv)]$S_{n,1,j}'=\emptyset$ and $S_{n,2,j}'=\{j-r,j\}$ for all
$n\geq0$ and $j\in\{r+1,\dots,2r-1\}$.
\end{enumerate}
If $z_{i}^{1}$ of Algorithm \ref{alg:pdt-space-alg} equals to $z_{i}^{1,0}$
of Algorithm \ref{alg:Dist-Dyk} for all $i\in\{1,\dots,r\}$, then
$z_{i}^{n}=z_{i}^{n,0}$ for all $n\geq0$ and $i\in\{1,\dots,r\}$. \end{prop}
\begin{proof}
We prove our result by induction. Consider the equalities\begin{subequations}\label{eq_m:ind-hyp}
\begin{eqnarray}
 &  & \begin{array}{c}
z_{i}^{k'}=z_{i}^{k',0}\mbox{ for all }i\in\{1,\dots,r\}\end{array}\label{eq:ind-hyp-1}\\
 & \mbox{ and } & \begin{array}{c}
x^{k'}=x_{0}-\frac{1}{r}\underset{i=1}{\overset{r}{\sum}}z_{i}^{k'}.\end{array}\label{eq:ind-hyp-2}
\end{eqnarray}
\end{subequations}For $k'=1$, \eqref{eq:ind-hyp-1} follows from
the induction hypothesis, and \eqref{eq:ind-hyp-2} follows from line
1 of Algorithm \ref{alg:pdt-space-alg}. We shall show that if \eqref{eq_m:ind-hyp}
holds for $k'=k$, then \eqref{eq_m:ind-hyp} holds for $k'=k+1$. 

First, $\{z_{i}^{k,1}\}_{i\in\{r+1,\dots,2r-1\}}$ are calculated
through \eqref{eq:outer-opt}. Proposition \ref{prop:direct-D1-D2}
shows that 
\begin{equation}
\begin{array}{c}
z_{i}^{k,1}\overset{\eqref{eq:outer-opt},(iii),\eqref{eq:z-j-conclusion}}{=}-\lambda_{i-r}\underset{i=1}{\overset{r}{\sum}}z_{i}^{k,0}\overset{(ii)}{=}-\frac{1}{r}\underset{i=1}{\overset{r}{\sum}}z_{i}^{k,0}.\end{array}\label{eq:pdt-pf-1}
\end{equation}
Next, for $i\in\{1,\dots,r-1\}$, solving the problems \eqref{eq_m:solve-inner-j}
with $S_{k,2,i+r}'\overset{(iv)}{=}\{i,i+r\}$ gives \begin{subequations}\label{eq_m:pdt-form-A}
\begin{eqnarray}
(z_{i}^{k,2},z_{i+r}^{k,2})= & \underset{(z_{i},z_{i+r})}{\arg\min} & \delta_{C_{i}}^{*}(z_{i})+\frac{1}{2r}\|r(z_{i+r}+x_{0})\|^{2}-\frac{1}{2r}\|rx_{0}\|^{2}\label{eq:pdt-form-1}\\
 & \mbox{s.t. } & z_{i}+z_{i+r}=z_{i}^{k,1}+z_{i+r}^{k,1}.\label{eq:pdt-form-2}
\end{eqnarray}
\end{subequations}One can calculate that 
\begin{eqnarray}
z_{i+r}+x_{0} & \overset{\eqref{eq:pdt-form-2}}{=} & \begin{array}{c}
z_{i}^{k,1}+z_{i+r}^{k,1}-z_{i}+x_{0}\overset{\eqref{eq:pdt-pf-1}}{=}z_{i}^{k,1}-\frac{1}{r}\underset{i'=1}{\overset{r}{\sum}}z_{i'}^{k,0}-z_{i}+x_{0},\end{array}\label{eq:z-i-plus-r-form}
\end{eqnarray}
and 
\begin{eqnarray}
\begin{array}{c}
z_{i}^{k,1}-\frac{1}{r}\underset{i'=1}{\overset{r}{\sum}}z_{i'}^{k,0}-z_{i}+x_{0}\end{array} & \!\!\overset{\eqref{eq:ind-hyp-2}}{=}\!\! & \begin{array}{c}
z_{i}^{k,1}+x^{k}-z_{i}\end{array}\label{eq:additional-chain}\\
 & \!\!\overset{\scriptsize\mbox{Alg. \ref{alg:calc-inner-opt}, line 5}}{=}\!\! & z_{i}^{k,0}+x^{k}-z_{i}\overset{\scriptsize\mbox{Alg. \ref{alg:pdt-space-alg}, line 4}}{=}\begin{array}{c}
u_{i}^{k+1}-z_{i}.\end{array}\nonumber 
\end{eqnarray}
so \eqref{eq_m:pdt-form-A} can be rewritten as 
\begin{eqnarray}
z_{i}^{k,2} & \overset{\eqref{eq:pdt-form-1}}{=} & \begin{array}{c}
\underset{z_{i}}{\arg\min}\,\,\delta_{C_{i}}^{*}(z_{i})+\frac{1}{2r}\|r(z_{i+r}+x_{0})\|^{2}\end{array}\label{eq:z-k-2-i}\\
 & \overset{\eqref{eq:z-i-plus-r-form},\eqref{eq:additional-chain}}{=} & \begin{array}{c}
\underset{z_{i}}{\arg\min}\,\,\delta_{C_{i}}^{*}(z_{i})+\frac{1}{2}\|z_{i}-u_{i}^{k+1}\|^{2}.\end{array}\nonumber 
\end{eqnarray}
Also, we have 
\begin{eqnarray}
\begin{array}{c}
-z_{r}-\underset{i=1}{\overset{r-1}{\sum}}z_{i}^{k,1}-\underset{i=r+1}{\overset{2r-1}{\sum}}z_{i}^{k,1}+x_{0}\end{array} & \overset{\eqref{eq:pdt-pf-1}}{=} & \begin{array}{c}
z_{r}^{k,1}-\frac{1}{r}\underset{i=1}{\overset{r}{\sum}}z_{i}^{k,1}+x_{0}-z_{r}.\end{array}\label{eq:z-r-term}
\end{eqnarray}
Next, solving the problem \eqref{eq:outer-opt} with $S_{k,2}\overset{(iii)}{=}\{r\}$
gives 
\begin{eqnarray}
z_{r}^{k,2} & \overset{\eqref{eq:outer-opt}}{=} & \!\!\!\!\!\begin{array}{c}
\underset{z_{r}}{\arg\min}\,\,\delta_{C_{r}}^{*}(z_{r})+\frac{1}{2r}\left\Vert r\left(-z_{r}-\underset{i=1}{\overset{r-1}{\sum}}z_{i}^{k,1}-\underset{i=r+1}{\overset{2r-1}{\sum}}z_{i}^{k,1}+x_{0}\right)\right\Vert ^{2}\end{array}\nonumber \\
 & \overset{\eqref{eq:z-r-term},\eqref{eq:additional-chain}}{=} & \!\!\!\!\!\begin{array}{c}
\underset{z_{r}}{\arg\min}\,\,\delta_{C_{r}}^{*}(z_{r})+\frac{1}{2r}\left\Vert z_{r}-u_{r}^{k+1}\right\Vert ^{2}.\end{array}\label{eq:z-k-2-r}
\end{eqnarray}
From \eqref{eq:z-k-2-i} and \eqref{eq:z-k-2-r}, we can see that
the forms for $z_{i}^{k,2}$ and $z_{r}^{k,2}$ are identical. Similar
to the relationships between \eqref{eq:first-primal} and \eqref{eq:dual-1},
these problems are the (Fenchel) dual problems to 
\begin{equation}
\min_{x}\delta_{C_{i}}(x)+\frac{1}{2}\|x-u_{i}^{k+1}\|^{2},\label{eq:proj-u-k-plus-1}
\end{equation}
which has primal solution $P_{C_{i}}(u_{i}^{k+1})$ and dual solution
$u_{i}^{k+1}-P_{C_{i}}(u_{i}^{k+1})$. (Specifically, one has $0\overset{\eqref{eq:proj-u-k-plus-1}}{\in}\partial\delta_{C_{i}}(P_{C_{i}}(u_{i}^{k+1}))+P_{C_{i}}(u_{i}^{k+1})-u_{i}^{k+1}$,
which gives $P_{C_{i}}(u_{i}^{k+1})\in\partial\delta_{C_{i}}^{*}(u_{i}^{k+1}-P_{C_{i}}(u_{i}^{k+1}))$,
and thus $u_{i}^{k+1}-P_{C_{i}}(u_{i}^{k+1})$ minimizes \eqref{eq:z-k-2-i}.)
Therefore, 
\[
z_{i}^{k,2}=u_{i}^{k+1}-P_{C_{i}}(u_{i}^{k+1})\overset{\scriptsize\mbox{Alg \ref{alg:pdt-space-alg}, lines 5,6}}{=}z_{i}^{k+1}.
\]
Combining with the fact that $z_{i}^{k+1,0}=z_{i}^{k,2}$, we have
$z_{i}^{k+1,0}=z_{i}^{k+1}$ for all $i\in\{1,\dots,r\}$ as needed.
To complete our induction, note that from Algorithm \ref{alg:pdt-space-alg},
we have 
\begin{eqnarray*}
x^{k+1} & \overset{\scriptsize\mbox{line 8}}{=} & \begin{array}{c}
\frac{1}{r}\underset{i=1}{\overset{r}{\sum}}x_{i}^{k+1}\overset{\scriptsize\mbox{line 6}}{=}\frac{1}{r}\underset{i=1}{\overset{r}{\sum}}(u_{i}^{k+1}-z_{i}^{k+1})\end{array}\\
 & \overset{\scriptsize\mbox{line 4}}{=} & \begin{array}{c}
\frac{1}{r}\underset{i=1}{\overset{r}{\sum}}(x^{k}+z_{i}^{k}-z_{i}^{k+1})\overset{\eqref{eq:ind-hyp-2}}{=}x_{0}-\frac{1}{r}\underset{i=1}{\overset{r}{\sum}}z_{i}^{k+1}.\end{array}
\end{eqnarray*}

\end{proof}

\section{Convergence}

In this section, we prove a convergence result for Algorithm \ref{alg:Dist-Dyk}
combined with Algorithm \ref{alg:calc-inner-opt}. Our result would
cover the case of the product space decomposition as well as the original
Dykstra's algorithm. 

Throughout this section, we make the following assumption on $g(\cdot)$
and $\{\lambda_{i}\}_{i=0}^{r-1}$. 
\begin{assumption}
\label{assu:g-and-lambda}Assume $m\geq0$, $\lambda_{i}=\frac{1}{m+1}$
for all $i\in\{0,\dots,m\}$, and $g(x)=\frac{m+1}{2}\|x-x_{0}\|^{2}$.
The term $\lambda_{0}g^{*}(\cdots)$ in \eqref{eq:outer-opt} becomes
\begin{equation}
\frac{1}{2}\left\Vert -\sum_{i\in S_{n,w}}z_{i}-\sum_{i\notin S_{n,w}}z_{i}^{n,w-1}+x_{0}\right\Vert ^{2}-\frac{1}{2}\|x_{0}\|^{2}.\label{eq:form-for-g}
\end{equation}

\end{assumption}
We define $v^{n,w}\in X$ and $x^{n,w}\in X$ to be\begin{subequations}\label{eq_m:from-10-13}
\begin{eqnarray}
v^{n,w} & := & \begin{array}{c}
\underset{i=1}{\overset{r+m}{\sum}}z_{i}^{n,w}\end{array}\label{eq:from-10}\\
\mbox{ and }x^{n,w} & := & \begin{array}{c}
x_{0}-v^{n,w}.\end{array}\label{eq:From-13}
\end{eqnarray}
\end{subequations}
\begin{assumption}
\label{assu:ABC}For each $n$ and $i$, there is an index $p(n,i)$
in $\{1,\dots,\bar{w}\}$ such that 
\begin{itemize}
\item [(A)]For all $i\in\{1,\dots,r+m\}$, 

\begin{itemize}
\item [(i)]$i\notin S_{n,w}$ and $i\notin S_{n,w,j}'$ for all $w\in\{p(n,i)+1,\dots,\bar{w}\}$
and $j\in\{r+1,\dots,r+m\}$, and
\item [(ii)]Either $i\in S_{n,p(n,i)}$, or $i\in S_{n,p(n,i),j}'$ for
some $j\in\{r+1,\dots,r+m\}$.
\end{itemize}
\item []This can be easily checked using line 5 of Algorithm \ref{alg:calc-inner-opt}
to lead to 
\begin{equation}
z_{i}^{n,p(n,i)}=z_{i}^{n,p(n,i)+1}=\cdots=z_{i}^{n,\bar{w}}\mbox{ for all }i\in\{1,\dots,r+m\}.\label{eq:z-p-equals-z-w}
\end{equation}

\item [(B)]If $i\in\{1,\dots,r\}$ and $i\in S_{n,p(n,i),j}'$ for some
(unique) $j\in\{r+1,\dots,r+m\}$ (including the case $i=j$).

\begin{itemize}
\item There is some $q(n,i)$ in $\{1,\dots,p(n,i)-1\}$, such that $j\in S_{n,q(n,i)}$,
and 
\begin{eqnarray*}
 &  & S_{n,p(n,i),j}'\cap S_{n,w}=\emptyset\mbox{ and }S_{n,p(n,i),j}'\cap S_{n,w,j'}'=\emptyset\\
 &  & \qquad\mbox{ for all }w\in\{q(n,i)+1,\dots,p(n,i)-1\}\mbox{ and }j'\in\{r+1,\dots,r+m\}.
\end{eqnarray*}

\end{itemize}
\item []This can be easily checked using line 5 of Algorithm \ref{alg:calc-inner-opt}
to lead to 
\begin{equation}
z_{i'}^{n,q(n,i)}=z_{i'}^{n,q(n,i)+1}=\cdots=z_{i'}^{n,p(n,i)-1}\mbox{ for all }i'\in S_{n,p(n,i),j}'.\label{eq:z-q-equal-z-p}
\end{equation}

\end{itemize}
\end{assumption}
Assumptions \ref{assu:g-and-lambda} and \ref{assu:ABC}(B) can be
further generalized, but we feel that they are enough to capture the
main ideas needed for more general cases. Moreover, one can check
that the classical Dykstra's algorithm and the product space formulation
satisfy these assumptions. Assumptions \ref{assu:ABC}(A) implies
that all the components of $z$ are updated in an iteration.

\subsection{Proof of convergence}
\begin{claim}
\label{claim:Fenchel-duality}For all $i\in S_{n,w}$, we have 
\begin{enumerate}
\item [(a)]$-x^{n,w}+\partial h_{i}^{*}(z_{i}^{n,w})\ni0$,
\item [(b)]$-z_{i}^{n,w}+\partial h_{i}(x^{n,w})\ni0$, and
\item [(c)]$h_{i}(x^{n,w})+h_{i}^{*}(z_{i}^{n,w})=\langle x^{n,w},z_{i}^{n,w}\rangle$. 
\end{enumerate}
\end{claim}
\begin{proof}
By taking the optimality conditions in \eqref{eq:outer-opt} with
respect to $z_{i}$ for $i\in S_{n,w}$, we have 
\begin{eqnarray*}
0 & \overset{\eqref{eq:outer-opt},\eqref{eq:form-for-g}}{\in} & \partial h_{i}^{*}(z_{i}^{n,w})+\sum_{i\in S_{n,w}}z_{i}^{n,w}+\sum_{i\notin S_{n,w}}z_{i}^{n,w-1}-x_{0}\\
 & \overset{\eqref{eq:inner-opt-2}}{=} & \partial h_{i}^{*}(z_{i}^{n,w})+\sum_{i=1}^{r}z_{i}^{n,w}-x_{0}\overset{\eqref{eq_m:from-10-13}}{=}\partial h_{i}^{*}(z_{i}^{n,w})-x^{n,w},
\end{eqnarray*}
so (a) holds. The equivalences of (a), (b) and (c) is standard. 
\end{proof}
Another rather standard and elementary result is as follows. We refer
to \cite{Pang_Dyk_spl} for its (short) proof.
\begin{prop}
\label{prop:primal-form}(On solving \eqref{eq:outer-opt}) If a minimizer
$z^{n,w}$ for \eqref{eq:outer-opt} exists, then the $x^{n,w}$ in
\eqref{eq:From-13} satisfies 
\begin{equation}
x^{n,w}=\begin{array}{c}
\underset{x\in X}{\arg\min}\underset{i\in S_{n,w}}{\sum}h_{i}(x)+\frac{1}{2}\bigg\| x-\bigg(x_{0}-\underset{i\notin S_{n,w}}{\sum}z_{i}^{n,w}\bigg)\bigg\|^{2}.\end{array}\label{eq:primal-subpblm}
\end{equation}
Conversely, if $x^{n,w}$ solves \eqref{eq:primal-subpblm} with the
dual variables $\{\tilde{z}_{i}^{n,w}\}_{i\in S_{n,w}}$ satisfying
\begin{equation}
\begin{array}{c}
\tilde{z}_{i}^{n,w}\in\partial h_{i}(x^{n,w})\mbox{ and }x^{n,w}-x_{0}+\underset{i\notin S_{n,w}}{\sum}z_{i}^{n,w}+\underset{i\in S_{n,w}}{\sum}\tilde{z}_{i}^{n,w}=0,\end{array}\label{eq:primal-optim-cond}
\end{equation}
then $\{\tilde{z}_{i}^{n,w}\}_{i\in S_{n,w}}$ solves \eqref{eq:outer-opt}. 
\end{prop}
For any $x\in X$ and $z\in X^{r+1}$, the analogue of \cite[(8)]{Gaffke_Mathar}
is 
\begin{eqnarray}
 &  & \begin{array}{c}
\frac{1}{2}\|x_{0}-x\|^{2}+\underset{i=1}{\overset{r+m}{\sum}}h_{i}(x)-F(z)\end{array}\label{eq:From-8}\\
 & \overset{\eqref{eq:dual}}{=} & \begin{array}{c}
\frac{1}{2}\|x_{0}-x\|^{2}+\underset{i=1}{\overset{r+m}{\sum}}[h_{i}(x)+h_{i}^{*}(z_{i})]-\left\langle x_{0},\underset{i=1}{\overset{r+m}{\sum}}z_{i}\right\rangle +\frac{1}{2}\left\Vert \underset{i=1}{\overset{r+m}{\sum}}z_{i}\right\Vert ^{2}\end{array}\nonumber \\
 & \overset{\scriptsize\mbox{Fenchel duality}}{\geq} & \begin{array}{c}
\frac{1}{2}\|x_{0}-x\|^{2}+\underset{i=1}{\overset{r+m}{\sum}}\langle x,z_{i}\rangle-\left\langle x_{0},\underset{i=1}{\overset{r+m}{\sum}}z_{i}\right\rangle +\frac{1}{2}\left\Vert \underset{i=1}{\overset{r+m}{\sum}}z_{i}\right\Vert ^{2}\end{array}\nonumber \\
 & = & \begin{array}{c}
\frac{1}{2}\left\Vert x_{0}-x-\underset{i=1}{\overset{r+m}{\sum}}z_{i}\right\Vert ^{2}\geq0.\end{array}\nonumber 
\end{eqnarray}

We now prove our convergence result.
\begin{thm}
\label{thm:convergence}(Convergence result) Suppose Assumptions \ref{assu:g-and-lambda}
and \ref{assu:ABC} hold. Consider the sequence $\{z^{n,w}\}_{{1\leq n<\infty\atop 0\leq w\leq\bar{w}}}\subset X^{r+m}$
generated by Algorithm \ref{alg:Dist-Dyk} with Algorithm \ref{alg:calc-inner-opt}
used to calculate \eqref{eq_m:inner-opt}. Suppose that 
\begin{itemize}
\item The value of \eqref{eq:first-primal} (i.e., the primal objective
value) is $\alpha$ and is finite, and the value of \eqref{eq:dual}
(i.e., the dual objective value) is $\beta$. 
\item $\|z^{n,\bar{w}}\|\in O(\sqrt{n})$.
\item Minimizers can be obtained for the problems \eqref{eq:outer-opt}
and \eqref{eq_m:solve-inner-j}.
\end{itemize}
The sequences $\{v^{n,w}\}_{{1\leq n<\infty\atop 0\leq w\leq\bar{w}}}\subset X$
and $\{x^{n,w}\}_{{1\leq n<\infty\atop 0\leq w\leq\bar{w}}}\subset X$
are then deduced from \eqref{eq_m:from-10-13}, and we have:
\begin{enumerate}
\item [(i)]The sum 
\[
\begin{array}{c}
\underset{n=1}{\overset{\infty}{\sum}}\underset{w=1}{\overset{\bar{w}}{\sum}}\Bigg[\|v^{n,w}-v^{n,w-1}\|^{2}+\underset{{j:r+1\leq j\leq r+m\atop S_{n,w,j}'\neq\emptyset}}{\sum}\|z_{j}^{n,w}-z_{j}^{n,w-1}\|^{2}\Bigg]\end{array}
\]
is finite, and $\{F(z^{n,\bar{w}})\}_{n=1}^{\infty}$ is nondecreasing,
where $F(\cdot)$ is as defined in \eqref{eq:dual}.
\item [(ii)]There is a constant $C$ such that $\|v^{n,w}\|^{2}\leq C$
for all $n\in\mathbb{N}$ and $w\in\{1,\dots,\bar{w}\}$. 
\item [(iii)]Let 
\begin{equation}
\begin{array}{c}
\gamma_{n}:=\underset{w=1}{\overset{\bar{w}}{\sum}}\bigg(\|v^{n,w}-v^{n,w-1}\|+\underset{{j:r+1\leq j\leq r+m\atop S_{n,w,j}'\neq\emptyset}}{\sum}\|z_{j}^{n,w}-z_{j}^{n,w-1}\|\bigg).\end{array}\label{eq:def-gamma}
\end{equation}
There exists a subsequence $\{v^{n_{k},\bar{w}}\}_{k=1}^{\infty}$
of $\{v^{n,\bar{w}}\}_{n=1}^{\infty}$ which converges to some $v^{*}\in X$
and that 
\begin{equation}
\lim_{k\to\infty}\gamma_{n_{k}}\sqrt{n_{k}}=0.\label{eq:lim-sum-norm-z}
\end{equation}

\item [(iv)]For all $i\in\{1,\dots,r+m\}$ and $n\in\mathbb{N}$, we can
find  $x_{i}^{n}\in\partial h_{i}^{*}(z_{i}^{n,\bar{w}})$ such that
$\|x_{i}^{n}-(x_{0}-v^{n,\bar{w}})\|\leq\gamma_{n}$. 
\item [(v)]For the $v^{*}$ in (iii), $x_{0}-v^{*}$ is the minimizer of
the primal problem (P1), and $\lim_{k\to\infty}F(z^{n_{k},\bar{w}})=\frac{1}{2}\|v^{*}\|^{2}+\sum_{i=1}^{r+m}h_{i}(x_{0}-v^{*})$. 
\end{enumerate}
The properties (i) to (v) in turn imply that $\lim_{n\to\infty}x^{n,\bar{w}}$
exists, and $x_{0}-v^{*}$ is the primal minimizer of \eqref{eq:first-primal}.\end{thm}
\begin{proof}
 We first remark on the proof of this result. The proof in \cite{Pang_Dyk_spl}
was adapted from \cite{Gaffke_Mathar}. Part (iv) is new, and arises
from considering \eqref{eq_m:inner-opt}. This also results in changes
to the statements of the other parts of the corresponding result in
\cite{Pang_Dyk_spl}. 

We first show that (i) to (v) implies the final assertion. For all
$n\in\mathbb{N}$ we have, from weak duality, 
\begin{equation}
\begin{array}{c}
F(z^{n,\bar{w}})\leq\beta\leq\alpha\leq\frac{1}{2}\|x_{0}-(x_{0}-v^{*})\|^{2}+\underset{i=1}{\overset{r+m}{\sum}}h_{i}(x_{0}-v^{*}),\end{array}\label{eq:weak-duality}
\end{equation}
hence $\beta=\alpha=\frac{1}{2}\|x_{0}-(x_{0}-v^{*})\|^{2}+h(x_{0}-v^{*})$,
and that $x_{0}-v^{*}=\arg\min_{x}\sum_{i=1}^{r+m}h_{i}(x)+\frac{1}{2}\|x-x_{0}\|^{2}$.
Since the values $\{F(z^{n,\bar{w}})\}_{n=1}^{\infty}$ are nondecreasing
in $n$, we have 
\[
\begin{array}{c}
\underset{n\to\infty}{\lim}F(z^{n,\bar{w}})=\frac{1}{2}\|x_{0}-(x_{0}-v^{*})\|^{2}+\underset{i=1}{\overset{r+m}{\sum}}h_{i}(x_{0}-v^{*}),\end{array}
\]
and (substituting $x=x_{0}-v^{*}$ in \eqref{eq:From-8}) 
\begin{eqnarray*}
 &  & \begin{array}{c}
\frac{1}{2}\|x_{0}-(x_{0}-v^{*})\|^{2}+h(x_{0}-v^{*})-F(z^{n,\bar{w}})\end{array}\\
 & \overset{\eqref{eq:From-8},\eqref{eq:from-10}}{\geq} & \begin{array}{c}
\frac{1}{2}\|x_{0}-(x_{0}-v^{*})-v^{n,\bar{w}}\|^{2}\end{array}\\
 & \overset{\eqref{eq:From-13}}{=} & \begin{array}{c}
\frac{1}{2}\|x^{n,\bar{w}}-(x_{0}-v^{*})\|^{2}.\end{array}
\end{eqnarray*}
Hence $\lim_{n\to\infty}x^{n,\bar{w}}$ is the minimizer of (P1). 

It remains to prove assertions (i) to (v).

\textbf{Proof of (i):} For $j\in\{r,\dots,r+m\}$, let $z^{n,w,j}\in X^{r+m}$
be the vector such that 
\[
z_{i}^{n,w,j}=\begin{cases}
z_{i}^{n,w} & \mbox{ if }i\in S_{n,w}\mbox{ or }i\in S_{n,w,j'}'\mbox{ for some }j'\leq j\\
z_{i}^{n,w-1} & \mbox{ otherwise.}
\end{cases}
\]
From the fact that $\{z_{i}^{n,w,r}\}_{i\in S_{n,w}}=\{z_{i}^{n,w}\}_{i\in S_{n,w}}$
is a minimizer for the outer problem \eqref{eq:outer-opt}, we have
\begin{equation}
\begin{array}{c}
F(z^{n,w,r})+\frac{1}{2}\bigg\|\underset{i\in S_{n,w}}{\sum}z_{i}^{n,w}-\underset{i\in S_{n,w}}{\sum}z_{i}^{n,w-1}\bigg\|^{2}\leq F(z^{n,w-1}).\end{array}\label{eq:decrease-1}
\end{equation}
Next, we note that for $j\in\{r+1,\dots,r+m\}$ such that $S_{n,w,j}'\neq\emptyset$,
solving the inner problems sequentially like in line 3 of Algorithm
\ref{alg:calc-inner-opt}, where each optimization problem has the
form \eqref{eq:transformed-dual}, gives 
\begin{equation}
\begin{array}{c}
F(z^{n,w,j})+\frac{1}{2}\bigg\|\underset{i\in S_{n,w,j}'\backslash\{j\}}{\sum}z_{i}^{n,w}-\underset{i\in S_{n,w,j}'\backslash\{j\}}{\sum}z_{i}^{n,w-1}\bigg\|^{2}\leq F(z^{n,w,j-1}).\end{array}\label{eq:decrease-2}
\end{equation}
In view of \eqref{eq:solve-inner-j-2}, we have 
\[
\begin{array}{c}
\underset{i\in S_{n,w,j}'\backslash\{j\}}{\sum}z_{i}^{n,w}-\underset{i\in S_{n,w,j}'\backslash\{j\}}{\sum}z_{i}^{n,w-1}=z_{j}^{n,w-1}-z_{j}^{n,w}.\end{array}
\]
Observe that $z^{n,w,r+m}=z^{n,w}$. We can combine \eqref{eq:decrease-1}
and \eqref{eq:decrease-2} to get 
\begin{equation}
\begin{array}{c}
F(z^{n,w})+\frac{1}{2}\|v^{n,w}-v^{n,w-1}\|^{2}+\underset{{j:r+1\leq j\leq r+m\atop S_{n,w,j}'\neq\emptyset}}{\sum}\frac{1}{2}\|z_{j}^{n,w-1}-z_{j}^{n,w}\|^{2}\leq F(z^{n,w-1}).\end{array}\label{eq:dual-obj-dec-recurrence}
\end{equation}
Next, $F(z^{n,\bar{w}})\leq\alpha$ by weak duality. The proof of
the claim follows from summing \eqref{eq:dual-obj-dec-recurrence}
over all $n$.

\textbf{Proof of (ii):} Substituting $x$ in \eqref{eq:From-8} to
be the primal minimizer $x^{*}$ and $z$ to be $z^{n,w}$, we have
\begin{eqnarray*}
 &  & \begin{array}{c}
\frac{1}{2}\|x_{0}-x^{*}\|^{2}+\underset{i=1}{\overset{r+m}{\sum}}h_{i}(x^{*})-F(z^{1,0})\end{array}\\
 & \overset{\scriptsize\mbox{part (i)}}{\geq} & \begin{array}{c}
\frac{1}{2}\|x_{0}-x^{*}\|^{2}+\underset{i=1}{\overset{r+m}{\sum}}h_{i}(x^{*})-F(z^{n,w})\end{array}\\
 & \overset{\eqref{eq:From-8}}{\geq} & \begin{array}{c}
\frac{1}{2}\left\Vert x_{0}-x^{*}-\underset{i=1}{\overset{r+m}{\sum}}z_{i}^{n,w}\right\Vert ^{2}\overset{\eqref{eq:from-10}}{=}\frac{1}{2}\|x_{0}-x^{*}-v^{n,w}\|^{2}.\end{array}
\end{eqnarray*}
The conclusion is immediate.

\textbf{Proof of (iii): } We first show that 
\begin{equation}
\begin{array}{c}
\underset{n\to\infty}{\liminf}\,\,\gamma_{n}\sqrt{n}=0.\end{array}\label{eq:root-n-dec}
\end{equation}
Seeking a contradiction, suppose instead that there is an $\epsilon>0$
and $\bar{n}>0$ such that if $n>\bar{n}$, then $\gamma_{n}\sqrt{n}>\epsilon$.
By the Cauchy Schwarz inequality, we have  
\[
\frac{\epsilon^{2}}{n}<\gamma_{n}^{2}\leq\bar{w}(1+m)\underset{w=1}{\overset{\bar{w}}{\sum}}\bigg(\|v^{n,w}-v^{n,w-1}\|^{2}+\underset{{r+1\leq j\leq r+m\atop S_{n,w,j}'\neq\emptyset}}{\sum}\|z_{j}^{n,w}-z_{j}^{n,w-1}\|^{2}\bigg).
\]
This contradicts the earlier claim in (i).

Through \eqref{eq:root-n-dec}, we find a sequence $\{v^{n_{k}}\}_{k=1}^{\infty}$
such that $\lim_{k\to\infty}\gamma_{n_{k}}\sqrt{n_{k}}=0$, and by
part (ii), we can assume $\lim_{k\to\infty}v^{n_{k}}$ exists, say
$v^{*}$. This completes the proof of (iii). 

\textbf{Proof of (iv): }

If $i\in S_{n,p(n,i)}$, then $z_{i}^{n,\bar{w}}\overset{\eqref{eq:z-p-equals-z-w}}{=}z_{i}^{n,p(n,i)}$,
and by Claim \ref{claim:Fenchel-duality}, we have $x^{n,p(n,i)}\in\partial h_{i}(z_{i}^{n,p(n,i)})$.
We also have 
\begin{equation}
\|x^{n,p(n,i)}-(x_{0}-v^{n,\bar{w}})\|\overset{\eqref{eq:From-13}}{=}\|v^{n,p(n,i)}-v^{n,\bar{w}}\|\overset{\eqref{eq:def-gamma}}{\leq}\gamma_{n}.\label{eq:conv-ineq-1}
\end{equation}
So in this case, $x_{i}^{n}$ can be chosen to be $x^{n,p(n,i)}$. 

Next, if $i\notin S_{n,p(n,i)}$, then $i\in S_{n,p(n,i),j}'$ for
some $j\in\{r+1,\dots,r+m\}$. We first consider the case where $i\in\{1,\dots,r\}$.
We claim that $x_{i}^{n}$ can be chosen to be 
\begin{equation}
\begin{array}{c}
x_{i}^{n}:=x_{0}-\underset{i'\in S_{n,p(n,i),j}'\backslash\{j\}}{\sum}z_{i'}^{n,p(n,i)}-\underset{i'\notin S_{n,p(n,i),j}'\backslash\{j\}}{\sum}z_{i'}^{n,q(i)}.\end{array}\label{eq:choose-x-n-i}
\end{equation}
We look at how $x_{i}^{n}$ is related to $x^{n,q(n,i)}$. Recall
that $x^{n,q(n,i)}$ is derived from $z^{n,q(n,i)}$ with $j\in S_{n,q(n,i)}$,
and $\{z_{i'}^{n,q(n,i)}\}_{i'\in S_{n,q(n,i)}}$ is a minimizer of
\[
\underset{z_{i'}:i'\in S_{n,q(n,i)}}{\min}\sum_{i'\in S_{n,q(n,i)}}h_{i'}^{*}(z_{i'})+\frac{1}{2}\left\Vert -\sum_{i'\in S_{n,q(n,i)}}z_{i'}-\sum_{i'\notin S_{n,q(n,i)}}z_{i'}^{n,q(n,i)-1}+x_{0}\right\Vert ^{2}.
\]
By looking at the variable $z_{j}^{n,q(n,i)}$ only in \eqref{eq:outer-opt},
we have 
\begin{equation}
z_{j}^{n,q(n,i)}\overset{\eqref{eq:outer-opt}}{=}\arg\min_{z_{j}}h_{j}^{*}(z_{j})+\underbrace{\frac{1}{2}\left\Vert -\sum_{i'\neq j}z_{i'}^{n,q(n,i)}-z_{j}+x_{0}\right\Vert ^{2}-\frac{1}{2}\|x_{0}\|^{2}}_{=(\frac{1}{r}g)^{*}(-\sum_{i'\neq j}z_{i'}^{n,q(n,i)}-z_{j})}.\label{eq:z-n-q-min}
\end{equation}
(Note that the problem in \eqref{eq:z-n-q-min} is equivalent to \eqref{eq:outer-opt}
fixed to only $z_{j}^{n,q(n,i)}$ through \eqref{eq:inner-opt-2}.)
Since 
\begin{equation}
\begin{array}{c}
h_{j}^{*}(\cdot)\overset{\scriptsize{\mbox{Assn \ref{assu:g-and-lambda}}}}{=}(\frac{1}{m+1}g)^{*}(\cdot)\overset{\scriptsize{\mbox{Assn \ref{assu:g-and-lambda}}}}{=}\frac{1}{2}\|\cdot+x_{0}\|^{2}-\frac{1}{2}\|x_{0}\|^{2},\end{array}\label{eq:h-j-star}
\end{equation}
we have $\nabla h_{j}^{*}(\cdot)=\cdot+x_{0}$, which gives $0\overset{\eqref{eq:z-n-q-min}}{=}(z_{j}^{n,q(n,i)}+x_{0})+(\sum_{i'=1}^{r+m}z_{i'}^{n,q(n,i)}-x_{0})$,
or
\begin{equation}
\begin{array}{c}
z_{j}^{n,q(n,i)}\overset{\eqref{eq:z-n-q-min}}{=}-\underset{i'=1}{\overset{r+m}{\sum}}z_{i'}^{n,q(n,i)}.\end{array}\label{eq:z-n-q}
\end{equation}
Next, recall \eqref{eq:h-j-star}. The set of variables $\{z_{i'}^{n,p(n,i)}\}_{i'\in S_{n,p(n,i),j}'\backslash\{j\}}$
is a minimizer of the problem \eqref{eq_m:solve-inner-j}, which can
be written through Proposition \ref{prop:reduced-pblm} as  
\begin{equation}
\begin{array}{c}
\underset{z_{i'}:i'\in S_{n,p(n,i),j}'\backslash\{j\}}{\min}\underset{i'\in S_{n,p(n,i),j}'\backslash\{j\}}{\sum}h_{i'}^{*}(z_{i'})+\frac{1}{2}\bigg\|\underset{i'\in S_{n,p(n,i),j}'}{\sum}z_{i'}^{n,p(n,i)-1}-\underset{i'\in S_{n,p(n,i),j}'\backslash\{j\}}{\sum}z_{i'}+x_{0}\bigg\|^{2}.\end{array}\label{eq:z-i-prime-in-S-prime-pblm}
\end{equation}
Now, 

\begin{equation}
\begin{array}{c}
\underset{i'\in S_{n,p(n,i),j}'}{\sum}z_{i'}^{n,p(n,i)-1}\overset{\eqref{eq:z-q-equal-z-p}}{=}\underset{i'\in S_{n,p(n,i),j}'}{\sum}z_{i'}^{n,q(n,i)}\overset{\eqref{eq:z-n-q}}{=}-\underset{i'\notin S_{n,p(n,i),j}'\backslash\{j\}}{\sum}z_{i'}^{n,q(n,i)},\end{array}\label{eq:manipulate-z-i-prime}
\end{equation}
Since $z_{i}^{n,p(n,i)}$ is a component of a minimizer of \eqref{eq:z-i-prime-in-S-prime-pblm},
we can use the optimality conditions there to get 
\begin{eqnarray*}
0 & \overset{\eqref{eq:z-i-prime-in-S-prime-pblm}}{\in} & \begin{array}{c}
\partial h_{i}^{*}(z_{i}^{n,p(n,i)})+\underset{i'\in S_{n,p(n,i),j}'\backslash\{j\}}{\sum}z_{i'}^{n,p(n,i)}-\underset{i'\in S_{n,p(n,i),j}'}{\sum}z_{i'}^{n,p(n,i)-1}-x_{0}\end{array}\\
 & \overset{\eqref{eq:manipulate-z-i-prime}}{=} & \begin{array}{c}
\partial h_{i}^{*}(z_{i}^{n,p(n,i)})+\underset{i'\in S_{n,p(n,i),j}'\backslash\{j\}}{\sum}z_{i'}^{n,p(n,i)}+\underset{i'\notin S_{n,p(n,i),j}'\backslash\{j\}}{\sum}z_{i'}^{n,q(i)}-x_{0}.\end{array}
\end{eqnarray*}
 The formula above gives $x_{i}^{n}\overset{\eqref{eq:choose-x-n-i}}{\in}\partial h_{i}^{*}(z_{i}^{n,p(n,i)})$.
Now 
\begin{eqnarray}
\|x_{i}^{n}-(x_{0}-v^{n,q(n,i)})\| & \overset{\eqref{eq:choose-x-n-i},\eqref{eq_m:from-10-13}}{=} & \begin{array}{c}
\bigg\|\underset{i'\in S_{n,p(n,i),j}'\backslash\{j\}}{\sum}\left(z_{i'}^{n,p(n,i)}-z_{i'}^{n,q(n,i)}\right)\bigg\|\end{array}\nonumber \\
 & \overset{\eqref{eq:z-q-equal-z-p}}{=} & \begin{array}{c}
\bigg\|\underset{i'\in S_{n,p(n,i),j}'\backslash\{j\}}{\sum}\left(z_{i'}^{n,p(n,i)}-z_{i'}^{n,p(n,i)-1}\right)\bigg\|\end{array}\nonumber \\
 & \overset{\eqref{eq:solve-inner-j-2}}{=} & \begin{array}{c}
\|z_{j}^{n,p(n,i)}-z_{j}^{n,p(n,i)-1}\|.\end{array}\label{eq:2nd-case-end}
\end{eqnarray}
We thus have 
\begin{equation}
\|x_{i}^{n}-(x_{0}-v^{n,\bar{w}})\|\leq\|x_{i}^{n}-(x_{0}-v^{n,q(n,i)})\|+\|v^{n,q(n,i)}-v^{n,\bar{w}}\|\overset{\eqref{eq:2nd-case-end},\eqref{eq:def-gamma}}{\leq}\gamma_{n}.\label{eq:conv-ineq-2}
\end{equation}
Lastly, we consider the case where $i\in\{r+1,\dots,r+m\}$ and $i\in S_{n,p(n,i),i}'$.
For this $i$, recall that $h_{i}^{*}(\cdot)=\frac{1}{2}\|\cdot+x_{0}\|^{2}-\frac{1}{2}\|x_{0}\|^{2}$.
Hence $\nabla h_{i}^{*}(\cdot)=\cdot+x_{0}$, which would mean that
$z_{i}^{n,p(n,i)}+x_{0}\in\partial h_{i}^{*}(z_{i}^{n,p(n,i)})$.
Let $x_{i}^{n}$ be $x_{0}+z_{i}^{n,p(n,i)}$. Then 
\begin{eqnarray}
\|x_{i}^{n}-(x_{0}-v^{n,\bar{w}})\| & = & \|z_{i}^{n,p(n,i)}+v^{n,\bar{w}}\|\label{eq:conv-ineq-3}\\
 & \overset{\eqref{eq:z-n-q}}{=} & \|z_{i}^{n,p(n,i)}-z_{i}^{n,q(n,i)}-v^{n,q(n,i)}+v^{n,\bar{w}}\|\nonumber \\
 & \leq & \|z_{i}^{n,p(n,i)}-z_{i}^{n,q(n,i)}\|+\|-v^{n,q(i)}+v^{n,\bar{w}}\|\nonumber \\
 & \overset{\eqref{eq:z-q-equal-z-p}}{=} & \|z_{i}^{n,p(n,i)}-z_{i}^{n,p(n,i)-1}\|+\|-v^{n,q(n,i)}+v^{n,\bar{w}}\|\overset{\eqref{eq:def-gamma}}{\leq}\gamma_{n}.\nonumber 
\end{eqnarray}

This ends of the proof of part (iv).

\textbf{Proof of (v):}  Let $x_{i}^{n}$ be as chosen in (iv). Since
$x_{i}^{n}\in\partial h_{i}^{*}(z_{i}^{n,\bar{w}})$, we have $h_{i}(x_{i}^{n})+h_{i}^{*}(z_{i}^{n,\bar{w}})=\langle x_{i}^{n},z_{i}^{n,\bar{w}}\rangle$.
From earlier results, we obtain 

\begin{eqnarray}
 &  & \begin{array}{c}
-\underset{i=1}{\overset{r+m}{\sum}}h_{i}(x_{0}-v^{*})\end{array}\label{eq:biggest-ineq}\\
 & \overset{\eqref{eq:From-8}}{\leq} & \begin{array}{c}
\frac{1}{2}\|x_{0}-(x_{0}-v^{*})\|^{2}-F(z^{n,\bar{w}})\end{array}\nonumber \\
 & \overset{\eqref{eq:dual}}{=} & \begin{array}{c}
\frac{1}{2}\|v^{*}\|^{2}+\underset{i=1}{\overset{r+m}{\sum}}h_{i}^{*}(z_{i}^{n,\bar{w}})-\langle x_{0},v^{n,\bar{w}}\rangle+\frac{1}{2}\|v^{n,\bar{w}}\|^{2}\end{array}\nonumber \\
 & = & \begin{array}{c}
\frac{1}{2}\|v^{*}\|^{2}+\underset{i=1}{\overset{r+m}{\sum}}\left(-h_{i}(x_{i}^{n})+\langle x_{i}^{n},z_{i}^{n,\bar{w}}\rangle\right)-\langle x_{0},v^{n,\bar{w}}\rangle+\frac{1}{2}\|v^{n,\bar{w}}\|^{2}\end{array}\nonumber \\
 & \overset{\eqref{eq:from-10}}{=} & \begin{array}{c}
\frac{1}{2}\|v^{*}\|^{2}-\frac{1}{2}\|v^{n,\bar{w}}\|^{2}+\underset{i=1}{\overset{r+m}{\sum}}\left(-h_{i}(x_{i}^{n})+\langle x_{i}^{n}-(x_{0}-v^{n,\bar{w}}),z_{i}^{n,\bar{w}}\rangle\right).\end{array}\nonumber 
\end{eqnarray}
In view of part (iv), we can choose $x_{i}^{n}$ to satisfy 
\begin{equation}
\langle x_{i}^{n}-(x_{0}-v^{n,\bar{w}}),z_{i}^{n,\bar{w}}\rangle\leq\|x_{i}^{n}-(x_{0}-v^{n,\bar{w}})\|\|z_{i}^{n,\bar{w}}\|\overset{\scriptsize\mbox{part (iv)}}{\leq}\gamma_{n}\|z_{i}^{n,\bar{w}}\|.\label{eq:intermediate-step}
\end{equation}
Recall the assumption that $\|z^{n,\bar{w}}\|\in O(\sqrt{n})$, and
from part (iii) that $\lim_{k\to\infty}\gamma_{n_{k}}\sqrt{n_{k}}=0$
and $\lim_{k\to\infty}v^{n,\bar{w}}=v^{*}$. We thus have 
\begin{equation}
\lim_{k\to\infty}\gamma_{n_{k}}\|z_{i}^{n_{k},\bar{w}}\|=0\mbox{ for all }i\in\{1,\dots,r+m\}.\label{eq:limit-gamma-z}
\end{equation}
We now look at \eqref{eq:biggest-ineq}. The first two terms in the
final line have limit $0$ as $k\nearrow\infty$. In view of \eqref{eq:intermediate-step}
and \eqref{eq:limit-gamma-z}, the limit of the last term in the final
line equals $\lim_{k\to\infty}-\sum_{i=1}^{r+m}h_{i}(x_{i}^{n_{k}})$
(we can limit to a subsequence so that this limit actually exists).
Part (iv) implies that $\lim_{k\to\infty}\gamma_{n_{k}}=0$, so for
all $i\in\{1,\dots,r+m\}$, we have 
\[
0\leq\lim_{k\to\infty}\|x_{i}^{n_{k}}-(x_{0}-v^{*})\|\overset{\scriptsize\mbox{part (iv)}}{\leq}\lim_{k\to\infty}\gamma_{n_{k}}=0.
\]
Thus $\lim_{k\to\infty}x_{i}^{n_{k}}=x_{0}-v^{*}$. The lower semicontinuity
of the functions $h_{i}(\cdot)$ implies that $\lim_{k\to\infty}-\sum_{i=1}^{r+m}h_{i}(x_{i}^{n_{k}})\leq-\sum_{i=1}^{r+m}h_{i}(x_{0}-v^{*})$.
Therefore this ends the proof of the result at hand. \end{proof}
\begin{rem}
(On $\|z^{n,\bar{w}}\|\in O(\sqrt{n})$) The level sets of the dual
problem may be unbounded. (See for example, \cite[page 9]{Han88}
and \cite[Section 4]{Gaffke_Mathar}.) In such a case the condition
$\|z^{n,\bar{w}}\|\in O(\sqrt{n})$ controls the rate of growth of
the dual variable $z^{n,\bar{w}}$ so that the proof of convergence
carries through. If Algorithm \ref{alg:Dist-Dyk} were run with Algorithm
\ref{alg:calc-inner-opt}, then a sufficient condition for $\|z^{n,\bar{w}}\|\in O(\sqrt{n})$
is that $|S_{n,w}|=1$ for all $n\in\mathbb{N}$ and $w\in\{1,\dots,\bar{w}\}$,
and that if $S_{n,w,j}'\neq\emptyset$, then $|S_{n,w,j}'\backslash\{j\}|=1$.
Depending on the structure of the functions $h_{i}(\cdot)$, it is
still possible to obtain $\|z^{n,\bar{w}}\|\in O(\sqrt{n})$ even
if $|S_{n,w}|>1$ or $|S_{n,w,j}'\backslash\{j\}|>1$. We refer to
\cite{Pang_Dyk_spl} for more details. Note that there are options
other than Algorithm \ref{alg:calc-inner-opt} to carry out \eqref{eq_m:solve-inner-j}.
For example, one can use the strategies mentioned in Subsection \ref{sub:general-strategies}.
A generalization of Theorem \ref{thm:convergence} for such situations
would either need $\|z^{n,\bar{w}}\|\in O(\sqrt{n})$ or a new method
of proof. 
\end{rem}

\subsection{Parallel computations satisfying Assumption \ref{assu:ABC}}

Consider a simple example where $r=2$, $m=2$, $\bar{w}=4$, and
that for all $n\geq1$, we have\begin{subequations}\label{eq-m:orig-calc}
\begin{equation}
S_{n,1}=\{3\},\quad S_{n,2}=\{1\},\quad S_{n,3}=\{2\},\quad S_{n,4}=\{4\}.\label{eq:orig-calc-1}
\end{equation}
If $S'_{n,w,j}=\emptyset$ for all $j\in\{3,4\}$ and $w\in\{1,2,3,4\}$,
then Assumption \ref{assu:ABC} is satisfied, and the convergence
theory of the earlier part of this section holds. If we have 
\begin{equation}
S'_{n,3,3}=\{1,3\}\mbox{ and }S'_{n,4,3}=\{2,3\},\label{eq:orig-calc-2}
\end{equation}
\end{subequations}with all other $S'_{n,j,w}=\emptyset$ instead,
then we can check that Assumption \ref{assu:ABC}(B) is not satisfied.
(Specifically, note that $p(n,3)=4$. But $3\notin S_{n,3}$ and $3\in S_{n,3,3}'$.)
But since calculations involving $S_{n,3,3}'$ and $S_{n,4,3}'$ do
not affect calculations involving $S_{n,3}$ and $S_{n,4}$, we can
move them to the next iteration counter $n+1$; Specifically, we have
$\bar{w}=6$, and for all $n\geq1$, 
\begin{eqnarray}
 &  & \tilde{S}_{n,1}=\tilde{S}_{n,2}=\emptyset,\quad,\tilde{S}_{n,3}=\{3\},\quad\tilde{S}_{n,4}=\{1\},\quad\tilde{S}_{n,5}=\{2\},\quad\tilde{S}_{n,6}=\{4\}\nonumber \\
 & \mbox{and} & \tilde{S}_{n+1,1,3}'=\{1,3\},\quad\tilde{S}_{n+1,2,3}=\{2,3\}.\label{eq:transferred-calc}
\end{eqnarray}
The calculations for \eqref{eq-m:orig-calc} and \eqref{eq:transferred-calc}
are the same, but transferring the calculations from $S'_{n,3,3}$
and $S'_{n,4,3}$ to $\tilde{S}'_{n+1,1,3}$ and $\tilde{S}'_{n+1,2,3}$
allows the convergence theory in the earlier part of this section
to go through.

\bibliographystyle{amsalpha}
\bibliography{../refs}

\end{document}